\pdfoutput=1
\documentclass[reqno,11pt]{amsart}
\numberwithin{equation}{section}

\usepackage[tt=false]{libertine}

\usepackage{amsmath}
\usepackage{amssymb}
\usepackage{mathrsfs}
\usepackage[varbb]{newpxmath}

\let\savedbigtimes\bigtimes
\let\bigtimes\relax
\usepackage{mathabx} 
\let\bigtimes\savedbigtimes

\usepackage{amsthm}
\newtheorem{thm}{Theorem}
\newtheorem{ppn}[thm]{Proposition}
\newtheorem{lem}[thm]{Lemma}
\newtheorem{cor}[thm]{Corollary}
\theoremstyle{dfn}
\newtheorem{dfn}{Definition}
\usepackage[margin=1in]{geometry}

\usepackage{bm}

\usepackage[usenames,dvipsnames]{xcolor}
\usepackage[colorlinks=true,  
	linkcolor=NavyBlue,
	citecolor=PineGreen,
	urlcolor=RedViolet]{hyperref}
\hypersetup{bookmarksopen=true}

\newcommand{\beq}{\begin{equation}}
\newcommand{\eeq}{\end{equation}}

\renewcommand{\epsilon}{\varepsilon}
\DeclareMathOperator{\Var}{Var}
\DeclareMathOperator{\Cov}{Cov}

\title{Locality of critical percolation on expanding graph sequences}
\date{\today}

\newcommand{\E}{\mathbb{E}}
\newcommand{\bP}{\mathbb{P}}
\newcommand{\Ind}[1]{\mathbf{1}\{#1\}}

\newcommand{\lwc}{\to_\textit{lwc}}

\newcommand{\GG}{\mathscr{G}}
\newcommand{\bd}{\bm{d}}

\renewcommand{\P}{\mathbb{P}}
\newcommand{\bL}{\mathbf{L}}

\DeclareMathOperator{\br}{br}

\newcommand{\bemph}[1]{\textbf{\textup{#1}}}

\author{Michael Ren$^{\star*}$}
\author{Nike Sun$^\star$}
\thanks{$^\star$Massachusetts Institute of Technology. $^*$University of Cambridge.}

\begin{document}

\begin{abstract} We study the locality of critical percolation on finite graphs: let $G_n$ be a sequence of finite graphs, converging locally weakly to a (random, rooted) infinite graph $G$. Consider Bernoulli edge percolation: does the critical probability for the emergence of an infinite component on $G$ coincide with the critical probability for the emergence of a linear-sized component on $G_n$? In this short article we give a positive answer provided the graphs $G_n$ satisfy an expansion condition, and the limiting graph $G$ has finite expected root degree. The main result of Benjamini, Nachmias, and Peres (2011), where this question was first formulated, showed the result assuming the $G_n$ satisfy a uniform degree bound and uniform expansion condition, and converge to a deterministic limit $G$. Later work of Sarkar (2021) extended the result to allow for a random limit $G$, but still required a uniform degree bound and uniform expansion for $G_n$. Our result replaces the degree bound on $G_n$ with the (milder) requirement that $G$ must have finite expected root degree. Our proof is a modification of the previous results, using a pruning procedure and the second moment method to control unbounded degrees.
\end{abstract}

\maketitle

\section{Introduction}

In this article we study the question of \bemph{locality of critical percolation}, as formulated by \cite{MR2773031}. Informally, if $G_n$ is a graph sequence converging locally weakly to the random graph $(G,\rho)$, then does Bernoulli bond percolation have the same critical probability on $G$ as on $G_n$? The main result of \cite{MR2773031} gives a positive answer in the case that the $G_n$ have a uniform Cheeger constant $h>0$ and a uniform degree bound $d<\infty$, and converges locally weakly to a deterministic limiting graph $(G,\rho)$. A later work \cite{MR4275958} extends this result to the more general case of a random limiting graph $(G,\rho)$, but still requires that the graphs $G_n$ satisfy uniform expansion and a uniform degree bound. In this article we extend the result of \cite{MR4275958} by further relaxing the degree condition: more precisely, we show that the degree restriction on $G_n$ can be replaced with the requirement that the limit $(G,\rho)$ must have finite expected root degree.

\subsection{Statement of main result}

To formally state our main result, we set some notations and review some standard definitions (see \cite{MR2354165}). If $v$ is a vertex in graph $H$, we let $B_R(v;H)$ be the $R$-neighborhood of $v$ in $H$; we can regard $B_R(v;H)$ as a graph rooted at $v$.

\begin{dfn}[space of rooted graphs]\label{d:rooted.gr}
Let $\GG_\star$ denote the space of connected locally finite rooted
graphs $(G,\rho)$, modulo rooted isomorphism. A metric on $\GG_\star$ is given by
	\beq\label{e:Gstar.metric}
	\bd\Big((G_1,\rho_1),(G_2,\rho_2)\Big)
	= \inf\bigg\{ \frac1{1+R}
	: R\ge0, B_R(\rho_1;G_1)
	\cong B_R(\rho_2;G_2)
	\bigg\}\in[0,1]\,.\eeq
The space $\GG_\star$ is separable and complete in this metric.
\end{dfn}

\begin{dfn}[local weak convergence]\label{d:lwc}
Suppose $G_n\equiv(V_n,E_n)$ is a sequence of finite graphs, and let $U_n$ be a uniformly random vertex in $G_n$. We say $G_n$ \bemph{converges locally weakly} to the (random) element $(G,\rho)\in\GG_\star$ if $B_R(U_n;G_n)$ converges in law to $(G,\rho)$ in the topology of $\GG_\star$. This  will be denoted 
$G_n\lwc (G,\rho)$; we will often write simply $G_n\lwc G$ with the understanding that $G$ is a (random) rooted graph.
\end{dfn}

\begin{dfn}[expanding graph sequences]\label{d:expand}
Given a finite graph $G=(V,E)$ and $0<\delta\le1/2$, define
	\[
	h_\delta(G)
	\equiv\min\bigg\{ 
	\frac{|E_G(A,V\setminus A)|}{|A|}
	: A\subseteq V,
	\delta\le\frac{|A|}{|V|}
		\le\frac12\bigg\}\,,
	\]
where $E_G(A,B)$ denotes the set of all edges in $G$ between the vertex sets $A$ and $B$. We say that a sequence $G_n$ of finite graphs is \bemph{expanding} if 
for all $\delta>0$ we have $\liminf_n h_\delta(G_n)>0$.
That is to say, for all $\delta>0$ there exist $n_\delta<\infty$ and $c_\delta>0$ such that $h_\delta(G_n)\ge c_\delta$ for all $n\ge n_\delta$.
\end{dfn}

\begin{dfn}[critical percolation]\label{d:crit.perc}
Let $G$ be an infinite graph, and let $\P_p$ denote the law of a Bernoulli bond percolation (or edge percolation) configuration $\omega=\omega(G)$ on $G$. Then
	\[
	p_c(G)
	= \sup_p \bigg\{ p\in[0,1]
	: \P_p\Big(\textup{$\omega$ has no infinite component}\Big)=1
	\bigg\}
	\]
is the \bemph{critical probability} for Bernoulli bond percolation on $G$.\end{dfn}

\begin{thm}\label{t:main}
Suppose that $G_n=(V_n,E_n)$ is an expanding sequence of finite graphs such that $G_n\lwc (G,\rho)$ for some locally finite random infinite rooted graph $(G,\rho)$ with $\mathbb E[\deg\rho]<\infty$. If $\omega(G_n)$ denotes edge percolation on $G_n$ with probability $p$, then the following hold:
\begin{itemize}
    \item[(i)] The critical edge percolation probability $p_c(G)$ is almost surely constant.
    \item[(ii)] If $p<p_c(G)$, then for any $\alpha>0$, the probability that the largest component in $\omega(G_n)$ has size more than $\alpha|V_n|$ tends to $0$ as $n\rightarrow\infty$.
    \item[(iii)] If $p>p_c(G)$, then there exists $\alpha>0$ such that the probability that the largest component in $\omega(G_n)$ has size more than $\alpha|V_n|$ tends to $1$ as $n\rightarrow\infty$.
\end{itemize}
Items (i)--(iii) can be summarized by saying that ``$p_c((G_n)_{n\ge0})=p_c(G)$.''
\end{thm}

\textbf{Theorem~\ref{t:main} removes the uniform degree bound assumption from \cite{MR2773031} and \cite{MR4275958}.} Further, 
\cite{MR2773031,MR4275958} assume \textbf{uniform expansion}
($h_\delta(G_n) \ge c_\delta$ for all $n\ge1$), while Theorem~\ref{t:main} uses only the weaker expansion assumption of Definition~\ref{d:expand}
($h_\delta(G_n) \ge c_\delta$ for all $n \ge n_\delta$). In fact, our argument for the unbounded degree setting involves pruning high-degree vertices, which can break a uniform expansion assumption.

Additionally, we remark that the locality result can fail without some kind of expansion condition. As an example,\footnote{We learned of this example from Elchanan Mossel.} let $h_n\to0$, and let $H$ be a graph on $d/h_n$ vertices, uniformly random subject to the requirement that all vertices have degree $2d$ except for one vertex $v_*(H)$ of degree $d$.
Let $H_1,\ldots,H_n$ be disjoint copies of $H$, and let $J$ be a random $d$-regular graph on the vertices $\{v(H_1),\ldots,v(H_n)\}$. Let $G_n$ be the union of $H_1,\ldots,H_n,J$, so it is a $(2d)$-regular graph. It has poor expansion: if we take $A=H_1 \cup \ldots \cup H_k$,
 then there are at most $kd$ edges between $A$ and its complement, so
	\[
	\frac{|E_G(A,V\setminus A)|}{|A|}
	\le \frac{kd}{kd/h_n} = h_n\,.
	\]
Then $G_n$ converges locally weakly to the $(2d)$-regular tree $G$, which has critical probability $p_c(G)=1/(2d-1)$. On the other hand, the sequence of finite graphs has critical probability 
$p_c((G_n)_{n\ge0}) \asymp 1/(d-1)$, as the emergence of a giant component on $G_n$ is determined by percolation on $J$.

\subsection{Related work} The locality question, as formulated by \cite{MR2773031}, was motivated by a related conjecture of Oded Schramm (see \cite[Conjecture~1.2]{MR2773031}) for the setting of infinite, vertex transitive graphs $G_n$. This conjecture has since been proven in a few special cases, e.g.,
for uniformly nonemenable transitive graphs \cite{song2014locality},
for abelian Cayley graphs \cite{MR3630298}, 
for transitive graphs with exponential growth \cite{MR4112717}, and
for transitive graphs with exponential growth \cite{contreras2022locality}.

Percolation on general finite graphs, initiated by \cite{MR2073175}, has also been well-studied. Beyond the existence of a percolation threshold for sequences of graphs \cite{MR2522948,MR3122364},
work has also been done to address the uniqueness of the giant component  \cite{MR2073175,easo2021supercritical} 
and scaling windows around the critical probability
\cite{MR2583058,MR2155704}. We remark that many of these results which give detailed descriptions of behavior above, below, and near the threshold are about specific graph sequences such as random regular graphs or rely on strong assumptions about the graphs such as having uniformly bounded degrees. On the other hand, the results from \cite{MR2773031} and our results apply more generally but only show the existence of a threshold for locally weakly convergent graph sequences.

In the case that $G_n$ converges locally weakly to a (random, rooted) \textbf{tree} $G$, more can be said about the percolation threshold: it is well known that $p_c(G)$ is the reciprocal of the tree \textbf{branching number} $\br(G)$ \cite{MR1062053}. More recently, physicists have proposed \cite{PhysRevLett.113.208702} that $p_c((G_n)_{n\ge0})$ can be estimated by the reciprocal of $\lambda^\textup{NB}(G_n)$, the top eigenvalue of the nonbacktracking matrix of $G_n$. Our main result Theorem~\ref{t:main}, for the case the local limit $G$ is a tree, implies 
$p_c((G_n)_{n\ge0})=1/\br(G)$. This naturally suggests that one approach for proving the conjecture of \cite{PhysRevLett.113.208702} would be to relate the asympotics of $\lambda(B_n)$ to the branching number of $G_n$.

\subsection{Organization} 
In Section~\ref{s:prelim} we prove some preliminary consequences of the local weak convergence assumption, and give our pruning argument. 
In Section~\ref{s:proof} we complete the proof of Theorem~\ref{t:main}.

\subsection*{Acknowledgements}  
We thank Elchanan Mossel, Lenka Zdeborov\'a, Jiaoyang Huang, and Fan Wei for helpful conversations. Research supported in part by
NSF CAREER grant DMS-1940092.

\section{Preliminaries}\label{s:prelim}

We fix some basic notation to be used throughout. In a graph $H$, let $d_H(u,v)$ be the distance between vertices $u$ and $v$; we will sometimes write simply $d(u,v)$ if $H$ is clear from context. Recall that $B_R(v;H)$ denotes the $R$-neighborhood of $v$ in $H$, i.e., the subgraph of $H$ induced by all vertices within distance $R$ of $v$. More generally, if $S$ is any subset of vertices in $H$, we write $B_R(S;H)$ for the subgraph of $H$ induced by all vertices within distance $R$ of $S$. We write $\partial B_R(v;H)$ for the set of vertices in $H$ at exactly distance $R$ from $v$.

We will also be working with a few different probability measures throughout the proof. Let $\P_p$ and $\P_{n,p}$ denote edge percolation on $G$ and $G_n$, respectively, with probability $p$. We use $\omega(G)$ and $\omega(G_n)$ to denote the percolation configurations. Let $\bL_n$ be the law of a uniform random vertex $U_n$ of $G_n$, and let $\mu$ be the law of $(G,\rho)$.

\subsection{Continuity lemma}

It follows directly from the definition of local weak convergence (Definition~\ref{d:lwc}) that if $f$ is any bounded continuous function on the space $\GG_\star$ (Definition~\ref{d:rooted.gr}), and $G_n\lwc (G,\rho)$, then 
	\beq\label{e:lwc.cts}
	\E_{\bL_n} f(G_n,U_n)
	\stackrel{n\to\infty}{\longrightarrow}
	\int f(G,\rho) \,d\mu\,.
	\eeq
An immediate consequence is the following:

\begin{lem}\label{l:localevent}
Suppose $f(G,\rho)$ is a bounded function on $\GG_\star$ that depends only on $B_R(\rho;G)$ for some finite $R$. Then $f$ is continuous on $\GG_\star$, so \eqref{e:lwc.cts} holds whenever $G_n\lwc (G,\rho)$.

\begin{proof}
Recall from Definition~\ref{d:rooted.gr} that a metric on $\GG_\star$ is given by $\bd$ from \eqref{e:Gstar.metric}. If 
	\[
	\bd\Big((G_1,\rho_1),(G_2,\rho_2)\Big) \le \frac1{1+R}
	\]
then $B_R(\rho_1;G_1)$ and $B_R(\rho_2;G_2)$ agree,
in which case the assumption on $f$ implies $f(G_1,\rho_1)=f(G_2,\rho_2)$. This shows that $f$ is (uniformly) continuous on $\GG_\star$.
\end{proof}
\end{lem}

\begin{cor}\label{c:localevent}
If $G_n\lwc (G,\rho)$, then for any measurable event $E$ we have
	\[
	\bL_n\otimes \P_{n,p}
	\bigg(
	\omega(B_R(U_n;G_n)) \in E
	\bigg)
	\stackrel{n\to\infty}{\longrightarrow}
	\mu\otimes\P_p
	\bigg(
	\omega(B_R(\rho;G)) \in E
	\bigg)\,.
	\]
\begin{proof} 
Define a function $f$ on $\GG_\star$ by setting
	\[
	f(G,\rho)
	= \P_p\bigg( \omega( B_R(\rho;G)  ) \in E\bigg)\,.
	\]
Then $f$ clearly depends only on $B_R(\rho;G)$, so the claim follows from Lemma~\ref{l:localevent}.
\end{proof}
\end{cor}

Lemma~\ref{l:localevent} and Corollary~\ref{c:localevent} will be used repeatedly throughout the proofs below.

\subsection{Pruning}

We will assume that the average degree of $G_n$ is uniformly bounded, i.e., that there exists a constant $d$ such that the average degree of $G_n$ is at most $d$ for all $n$. In the case that this does not hold, the result can be recovered from the former case via a pruning argument, which we now give. The main result of this subsection is the following:

\begin{ppn}\label{p:prune}
Suppose $(G,\rho)$ is a (random) infinite rooted graph, locally finite, with $\E(\deg\rho)<\infty$. Suppose $G_n$ is an expanding sequence of finite graphs with $G_n\lwc (G,\rho)$. Then, for any $\epsilon>0$,
 we can choose a sequence of subgraphs $\bar{G}_n$ of $G_n$ such that
(i) each $\bar{G}_n$ has average degree at most $\E(\deg\rho)+\epsilon$;
and (ii) the $\bar{G}_n$ also form an expanding sequence, with $\bar{G}_n\lwc(G,\rho)$.
\end{ppn}

\begin{lem}\label{l:prune}
Suppose $(G,\rho)$ is a (random) infinite rooted graph, locally finite, with $\E(\deg\rho)<\infty$. Suppose $G_n$ is an expanding sequence of finite graphs with $G_n\lwc (G,\rho)$. Let $k_n$ be any sequence of integers with $k_n\to\infty$, and let $\bar{G}_n$ be the subgraph of $G_n$ formed by removing all edges incident to vertices of degree at least $k_n$. Then $\bar{G}_n$ is also an expanding sequence with $\bar{G}_n\lwc (G,\rho)$.

\begin{proof}
Let $S_n$ denote the subset of all vertices in $G_n$ with degree at least $k_n$. We first claim that for any fixed $R<\infty$,
	\beq\label{e:Sn.nbd.bound}
	\lim_{n\to\infty}\frac{|B_R(S_n;G_n)|}{|V_n|}=0\,.
	\eeq
Indeed, if a vertex is within distance $R$ of $S_n$, then it must have at least $k_n$ vertices in its $(R+1)$-neighborhood, so
	\[\frac{|B_R(S_n;G_n)|}{|V_n|}
	\le
	\bL_n\bigg(
		\Big|B_{R+1}(U_n;G_n)\Big|\ge k_n\bigg)\,.\]
Since $k_n\to\infty$ by assumption, it follows that for any finite $k$ we have
	\[
\limsup_{n\to\infty}
	\bL_n\bigg(
		\Big|B_{R+1}(U_n;G_n)\Big|\ge k_n\bigg)
	\le \limsup_{n\to\infty}
	\bL_n\bigg(
		\Big|B_{R+1}(U_n;G_n)\Big|\ge k\bigg)
	\]
Finally, by
the assumption $G_n\lwc(G,\rho)\sim\mu$ together with Lemma~\ref{l:localevent}, we have
	\[
	\lim_{n\to\infty}\bL_n\bigg(
		\Big|B_{R+1}(U_n;G_n)\Big|\ge k\bigg)
	=\mu\bigg(\Big|B_{R+1}(\rho,G)\Big|\ge k\bigg)\,.
	\]
Combining the above calculations gives
	\[\limsup_{n\to\infty}\frac{|B_R(S_n;G_n)|}{|V_n|}
	\le \mu\bigg(\Big|B_{R+1}(\rho,G)\Big|
	\ge k\bigg)
	\]
for any finite $k$. In the above, the right-hand side can be made arbitrarily small by taking $k\to\infty$, while the left-hand side does not depend on $k$, so this proves the claim \eqref{e:Sn.nbd.bound}.

We next argue that $\bar{G}_n\lwc (G,
\rho)$. The graphs $G_n$ and $\bar{G}_n$ have the same vertex set, so 
if $U_n$ is a uniformly random vertex in $G_n$, then it is also a uniformly random vertex in $\bar{G}_n$. Moreover, it is clear that the subgraphs $B_R(U_n;G_n)$ and $B_R(U_n;\bar{G}_n)$ are the same as long as $U_n$ does not belong to $B_R(S_n;G_n)$. We see from \eqref{e:Sn.nbd.bound} that $U_n$ lies in $B_R(S_n;G_n)$ with probability $o_n(1)$, so we conclude that
$B_R(U_n;\bar{G}_n)$ also converges in law
to $B_R(\rho;G)$. This proves that $\bar{G}_n\lwc (G,
\rho)$, as claimed.

We now argue that $\bar{G}_n$ is an expanding sequence. Given $\delta>0$, let $A\subseteq V_n$ with $\delta\le |A|/|V_n|\le1/2$. Since the original sequence $G_n$ is assumed to be expanding, there exists $c_\delta>0$ such that
	\[
	\frac{|E_{G_n}(A,V_n\setminus A)|}{|A|}
	\ge c_\delta
	\]
as long as $n\ge n_\delta$. The graph $\bar{G}_n$ differs from $G_n$ only in the deletion of edges incident to $S_n$, and the total number of deleted edges can be upper bounded by $|B_1(S_n;G_n)|$. It follows using \eqref{e:Sn.nbd.bound} that
	\[
	\frac{|E_{\bar{G}_n}(A,V_n\setminus A)|}{|A|}
	\ge
	\frac{|E_{{G}_n}(A,V_n\setminus A)|}{|A|}
	-\frac{|B_1(S_n;G_n)|}{|V_n|}
	\frac{|V_n|}{|A|}
	\ge 
	\frac{|E_{{G}_n}(A,V_n\setminus A)|}{|A|}
	- \frac{o_n(1)}{\delta}\,,
	\]
which is at least $c_\delta/2$ for $n$ large enough.
This shows that $\bar{G}_n$ is also an expanding sequence, and this concludes the proof of the lemma. 
\end{proof}
\end{lem}

\begin{lem}\label{l:integrability}
Let $D_n$ be a nonnegative integer-valued random variable converging in law to an integrable random variable $D$. Then, for any $\epsilon>0$,
it is possible to choose a sequence $k_n\to\infty$ such that $\E(D_n; D_n < k_n) \le \E D + \epsilon < \infty$.

\begin{proof}
First we note that for any finite integer $k$,
	\[
	\E(D; D < k)
	=\sum_{\ell\ge1} \bP\bigg(
	D\Ind{D< k} \ge\ell\bigg)
	=\sum_{\ell=1}^{k-1}\bP( \ell \le D < k)
	\le\sum_{\ell=1}^{k-1}
		\bP(D\ge \ell)\,.
	\]
It follows from the assumption that for any finite $\ell$ we have $\bP(D_n\ge\ell) \to \bP(D\ge\ell)$ as $n\to\infty$, and consequently
$\bP(D_n\ge\ell) \le \bP(D\ge\ell) + \epsilon/2^\ell$ for all $n\ge m_\ell$, where $m_\ell$ is a finite integer depending on $\ell$ and $\epsilon$. We then set $k_n\equiv n$ if $\bP(D_n\ge\ell) \le \bP(D\ge\ell) + \epsilon/2^\ell$ for all $\ell\ge1$ and otherwise
	\[
	k_n \equiv \min\bigg\{\ell\ge1 :
	\bP(D_n\ge\ell) > \bP(D\ge\ell) + \frac{\epsilon}{2^\ell} 
	 \bigg\}
	\ge \min\bigg\{\ell\ge1 :
		m_\ell > n\bigg\}.
	\]
Our chosen $k_n$ must tend to infinity as $n\to\infty$ since the $m_\ell$ are finite.
The definition of $k_n$ implies
	\[
	\E(D_n; D_n < k_n)
	\le \sum_{\ell=1}^{k_n-1} \bP(D_n\ge \ell)
	\le\sum_{\ell=1}^{k_n-1}\bigg\{
	\bP(D\ge \ell) + \frac{\epsilon}{2^\ell} 
	\bigg\}
	\le \E D + \epsilon\,,
	\]
as required. This finishes the proof.
\end{proof}
\end{lem}

\begin{proof}[Proof of Proposition~\ref{p:prune}]
This follows immediately by combining  Lemmas~\ref{l:prune} and \ref{l:integrability}.
\end{proof}

\section{Proof of Theorem~\ref{t:main}}
\label{s:proof}

In this section we give the proof of Theorem~\ref{t:main}.
We prove parts (i), (ii), and (iii) respectively
in \S\ref{ss:rerooting.inv}, \S\ref{ss:subcrit}, and \S\ref{ss:supercrit}.
As noted above, our proofs are based on arguments of
\cite{MR2773031,MR4275958}, with modifications to handle unbounded degrees. With the aim of being self-contained, we present a full proof below, but will repeatedly point out similarities and differences with the prior works.

\subsection{Rerooting invariant functions}
\label{ss:rerooting.inv}

The proof of Theorem~\ref{t:main} part~(i), after applying Proposition~\ref{p:prune}, is very similar to the proof of the corresponding result in \cite{MR4275958}, but we have made some simplifications to the main lemma. As in \cite{MR4275958}, the proof shows that any rerooting invariant function of $G$ is almost surely constant, so the local weak limit of an expanding sequence of finite graphs with finite expected root degree is in fact an extremal unimodular graph (see Corollary~\ref{c:extremal} below).

\begin{ppn}\label{p:sarkar}
Suppose that $\mathcal G_1$ and $\mathcal G_2$ are disjoint subsets of $\GG_\star$ such that each $\mathcal{G}_i$ is closed under rerooting. Then, for any compact subsets $H_i\subseteq\mathcal G_i$ and any finite integer $K$, there exists a finite integer $R$ such that for any $(G_1,\rho_1)\in H_1$ and $(G_2,\rho_2)\in H_2$, we have that $B_R(u_1,G_1)$ is not isomorphic to $B_R(u_2,G_2)$ for all $u_1\in B_K(\rho_1,G_1),u_2\in B_K(\rho_2,G_2)$.
\end{ppn}

In other words, given a rooted graph from one of $\mathcal G_1$ and $\mathcal G_2$, we can determine which set it came from by looking at the $R$-neighborhood of an arbitrary vertex in the $K$-neighborhood of the root. This was proven in \cite{MR4275958} by considering finite open covers of the $H_i$. We now give a conceptually cleaner proof of this fact.

\begin{lem}\label{l:reroot}
Recall the metric $\bd((G_1.\rho_1),(G_2,\rho_2))$ on $\GG_\star$ defined by \eqref{e:Gstar.metric}. For any finite $K$, let $D_K:\GG_\star\times\GG_\star\to[0,1]$ be defined by
	\[
	D_K\Big((G_1,\rho_1),(G_2,\rho_2)\Big)
	\equiv \min\bigg\{
	\bd\Big( (G_1,u_1),(G_2,u_2)\Big)
	: u_i \in B_K(\rho_i;G_i)
	\bigg\}\,,
	\]
where $(G_i,u_i)$ denotes the graph $(G_i,\rho_i)$ rerooted at $u_i$.
The function $D_K$ is continuous.

\begin{proof}
Since the graphs $G_i$ are locally finite, the definition of $D_K$ is a minimum over finitely many quantities. Therefore we can choose $u_{i\star}\in B_K(\rho_i;G_i)$ such that
	\beq\label{e:choose.u.i}
	D_K\Big((G_1,\rho_1),(G_2,\rho_2)\Big)
	= \bd\Big( (G_1,u_{1\star}),(G_2,u_{2\star})\Big)\,.
	\eeq
Now suppose that $(H,\sigma)$ is any element of $\GG_\star$ that lies within distance $\delta=1/(1+R)$ of $(G_2,\rho_2)$, where we assume $R\ge K$. From the definition \eqref{e:Gstar.metric}, it implies that there is a rooted isomorphism $\varphi$
that maps $B_R(\rho_2;G_2)$ to $B_R(\sigma;H)$.
Let $v_{2\star}\equiv\varphi(u_{2\star})$. We then have
	\begin{align*}
	&D_K\Big((G_1,\rho_1),(H,\sigma)\Big)
	\le
	\bd\Big( (G_1,u_{1\star}),(H,v_{2\star})\Big)  \\
	&\qquad\le
	\bd\Big( (G_1,u_{1\star}),(G_2,u_{2\star})\Big)
	+\bd\Big((G_2,u_{2\star}),(H,v_{2\star})\Big)\\
	&\qquad\le D_K\Big((G_1,\rho_1),(G_2,\rho_2)\Big)
	+ \frac1{1+(R-K)}
	\end{align*}
where the first step is by the definition of $D_K$,
the second step is by the triangle inequality,
and the last step is by \eqref{e:choose.u.i} together with
the observation that $B_{R-K}(u_{2\star};G_2)$ must be isomorphic to 
$B_{R-K}(u_{2\star};H)$. By exchanging the roles of $(G_2,\rho_2)$
and $(H,\sigma)$ we must also have
	\[D_K\Big((G_1,\rho_1),(G_2,\rho_2)\Big)
	\le D_K\Big((G_1,\rho_1),(H,\sigma)\Big)
	+ \frac1{1+(R-K)}\,,
	\]
and combining the last two inequalities gives
	\[
	\bigg| D_K\Big((G_1,\rho_1),(G_2,\rho_2)\Big)
	-D_K\Big((G_1,\rho_1),(H,\sigma)\Big)\bigg|
	\le \frac1{1+(R-K)}\,.
	\]
This shows that $D_K:\GG_\star\times\GG_\star\to[0,1]$ is uniformly continuous in the second coordinate. By symmetry it is also uniformly continuous in the first coordinate, which proves the claim.
\end{proof}
\end{lem}

\begin{proof}[Proof of Proposition~\ref{p:sarkar}]
By the assumption that the $\mathcal{G}_i$ are disjoint and closed under rerooting, the function $D_K$ must be strictly positive on
$\mathscr{G}_1\times\mathscr{G}_2$, and hence also on $H_1\times H_2$.
Since $D_K$ is continuous by Lemma~\ref{l:reroot}, we conclude that
the minimum of $D_K$ on the compact set $H_1\times H_2$ must be lower bounded by some $\delta>0$. Then, as long as we have $\delta>1/(1+R)$, it follows from the definition of $D_K$ that for any $(G_i,\rho_i)\in H_i$,
the graphs
$B_R(u_1;G_1)$ and $B_R(u_2;G_2)$ must be non-isomorphic
for all $u_i\in B_K(\rho_i;G_i)$.
\end{proof}

As in \cite{MR4275958}, the rest of the proof proceeds by showing that if $p_c(G)$ can take on different values, then the root vertices in $G_n$ that produce these different values in the limit can be connected by a short path, contradicting Proposition~\ref{p:sarkar}. This relies crucially on the expansion condition, which will be used in the application of the following classical result:

\begin{thm}[Menger's theorem]\label{t:menger}
In a finite graph $G=(V,E)$, for any two disjoint subsets $A,B\subseteq V$,
the minimum size of an $A$-$B$ edge cut equals the maximum number of pairwise edge-disjoint paths from $A$ to $B$.
\end{thm}

\begin{cor}\label{c:bdd.deg.constant.pc}
Suppose $G_n$ is an expanding sequence of finite graphs with $G_n\lwc(G,\rho)$, and such that each $G_n$ has average degree at most $d<\infty$. Then $p_c(G)$ is almost surely constant.

\begin{proof}
This proof is essentially identical to that of \cite[Theorem~1.1]{MR4275958}. The uniform degree bound is only used to bound the number of edges in the graph, so the argument can be directly adapted to when there is a uniform average degree bound. 

Suppose to the contrary that $p_c(G)$ is not almost surely constant. Then there exist $0\le a<b\le1$ such that, if we define 
	\[\mathcal{G}^1=\Big\{G:p_c(G)\le a\Big\}\,,\quad
	\mathcal{G}^2=\Big\{G:p_c(G)\ge b\Big\}\,,\]
then we will have a positive constant $p_0$ such that
	\[
	\min\Big\{ \mu(\mathcal{G}^1),\mu(\mathcal{G}^2)\Big\}\ge p_0\,.
	\]
It is clear that $\mathcal{G}^1$ and $\mathcal{G}^2$ are disjoint, and invariant under rerooting. Recall that $\GG_\star$ is a Polish space (see Definition~\ref{d:rooted.gr}), and that all Borel measures on Polish spaces are inner regular. It follows that there exist compact subsets $H^i\subseteq\mathcal{G}^i$ such that 
	\beq\label{e:choice.of.H}
	\min\Big\{ \mu(H^1),\mu(H^2)\Big\}
	\ge \frac{p_0}{2}\,.
	\eeq
Recall that $G_n$ is assumed to be an expanding sequence in the sense of Definition~\ref{d:expand}, and let $h$ be a positive constant such that
	\beq\label{e:lim.exp}
	\frac12\liminf_{n\to\infty} h_{p_0/4}(G_n) \ge h\,.
	\eeq
Recall that $G_n$ is also assumed to have average degree at most $d$, and define
	\[K=\frac{4d}{hp_0}\,.\]
Now apply Proposition~\ref{p:sarkar} with $H^1$, $H^2$, and $K$ as defined above, and let $R$ be the integer that results from the conclusion of the proposition. If we define the events
	\[
	E^i\equiv \bigg\{(G,\rho) : 
	B_{R+K}(\rho;G)
	\cong B_{R+K}(\rho^i;G^i)
	\textup{ for some } (G^i,\rho^i)\in H^i\bigg\}\,,
	\]
then $E^1$ and $E^2$ are disjoint, because Proposition~\ref{p:sarkar} guarantees in particular that
$B_R(\rho^1;G^1)\not\cong B_R(\rho^2;G^2)$
for all $(G^i,\rho^i)\in H^i$. Let $A^i$ be the set of vertices in $G_n$ such that $(G_n,v)$ belongs to event $E^i$. Since the events $E^1$ and $E^2$ are disjoint, the vertex sets $A^1$ and $A^1$ are disjoint. Since membership in $E^i$ depends only on the $(R+K)$-neighborhood of the root vertex, Lemma~\ref{l:localevent} gives
	\[
	\lim_{n\to\infty} \frac{|A^i|}{|V_n|}
	= \lim_{n\to\infty} \bL_n\Big( (G_n,U_n)\in E^i\Big)
	= \mu\Big( (G,\rho)\in E^i\Big) 
	\ge \mu(H^i)
	\stackrel{\eqref{e:choice.of.H}}{\ge} \frac{p_0}{2} >0\,.
	\]
Thus, for large enough $n$ we have $|A^i|/|V_n| \ge p_0/4$. It follows from the expansion assumption \eqref{e:lim.exp} that, again for $n$ large enough, the minimum edge cut separating $A^1$ from $A^2$ has at least $\kappa_n = hp_0|V_n|/4$ edges. It follows by Menger's theorem (Theorem~\ref{t:menger}) that there are at least $\kappa_n$ edge-disjoint paths joining $A^1$ to $A^2$. Since $G_n$ has at most $|V_n|d/2$ edges in total, at least half of the paths must have length at most
	\[\frac{|V_n|d/2}{\kappa_n/2} = 
	\frac{4d}{h p_0} = K\,.
	\]
However, even the existence of a single such path joining
$v^1\in A^1$ to $v^2\in A^2$ results in a contradiction: it implies 
that the $R$-neighborhood of $v^2$ lies inside the $(R+K)$-neighborhood of $v^1$ in $G_n$. The definition of $A^i$ implies $B_{R+K}(v^i;G_n)\cong B_{R+K}(\rho^i;G^i)$ for some $(G^i,\rho^i)\in H_i$. This contradicts the definition of $R$ and concludes the proof.
\end{proof}
\end{cor}

We remark that the above argument implies the result of \cite[Theorem~1.1]{MR4275958} under the milder assumptions of Theorem~\ref{t:main}. Any local weak limit of a sequence of finite graphs is a random rooted graph whose measure is unimodular (see \cite[Definition~2.1]{MR2354165}). The space of unimodular probability measures on $\GG_\star$ is convex, and its extreme points are said to be \bemph{extremal} or \bemph{ergodic}. A unimodular measure $\mu$ on $\GG_\star$ is extremal if and only if $\mu(E)\in\{0,1\}$ for any event $E$ that is invariant under non-rooted isomorphisms (\cite[Theorem~4.7]{MR2354165}). It then follows from the above that the local weak limit of an expanding sequence of finite graphs is extremal if its root has finite expected degree:

\begin{cor}\label{c:extremal}
Under the conditions of Theorem~\ref{t:main}, $G$ is an extremal unimodular random graph.
\end{cor}

\begin{proof}[Proof of Theorem~\ref{t:main} part (i)]
By Proposition~\ref{p:prune}, there exists a expanding sequence of subgraphs $\bar G_n$ of $G_n$ such that each $\bar G_n$ has average degree at most $\E(\deg\rho)+1$ and $\bar G_n\lwc(G,\rho)$. Applying Corollary~\ref{c:bdd.deg.constant.pc} to $\bar G_n$ gives the result.
\end{proof}

\subsection{Subcritical percolation regime}
\label{ss:subcrit}

In this subsection we prove Theorem~\ref{t:main} part (ii), which says in short that $p_c(G_n) \ge p_c(G)$. The result is a straightforward consequence of Lemmas~\ref{l:localevent} and Corollary~\ref{c:localevent}. Let us point out that for this part of the proof we work with the original sequence $G_n$, without pruning, so we do not assume that $G_n$ has bounded average degree. The reason for this is that (\textit{a~priori}) we only know that $p_c(\bar{G}_n) \ge p_c(G_n)$, so it is not sufficient to only prove $p_c(\bar{G}_n) \ge p_c(G)$.

Recall that we let $\omega=\omega(G)$ denote an edge percolation configuration on graph $G$; we also view $\omega$ as a (random) subgraph of $G$. Let $C(v)\equiv C(v;\omega)$ denote the connected component of $\omega$ that contains $v$. For any vertex subsets $A$ and $B$, let $A\leftrightarrow_\omega B$
 indicate that $A$ and $B$ are connected by a path using only edges in $\omega$. We write simply $A\leftrightarrow B$ when $\omega$ is clear from context.

\begin{proof}[Proof of Theorem~\ref{t:main} part (ii)]
The proof of \cite[Corollary~1.2]{MR4275958}
(which is based on the proof of \cite[Theorem~1.3]{MR2773031}) applies here verbatim. To be self-contained, we also give the argument here.
Under the assumptions of Theorem~\ref{t:main}, we will show that for any $p<p_c(G)$ and any $\alpha>0$, the probability that the largest component in $\omega(G_n)$ contains more than $\alpha|V_n|$ vertices tends to zero as $n\to\infty$. Let $C_n$ denote the vertices in the largest component of $\omega(G_n)$. For any $R>0$, we have the chain of inequalities
	\begin{align}\nonumber
    &\alpha\P_{n,p}\Big(|C_n|\ge\alpha|V_n|\Big)\\
    &\qquad\le\bL_n\otimes\P_{n,p}\Big(
    	|C_n|\ge\alpha|V_n|, U_n\in C_n\Big) \nonumber\\
    &\qquad\le\bL_n\otimes\P_{n,p}\Big(|C(U_n)|\ge\alpha|V_n|\Big)
    	\nonumber\\
    &\qquad\le\bL_n\otimes\P_{n,p}
    	\Big(|C(U_n)|\ge|B_R(U_n,G_n)|\Big)
	+\bL_n\Big(|B_R(U_n,G_n)|\ge\alpha|V_n|\Big) \nonumber\\
    &\qquad\le
    \bL_n\otimes\P_{n,p}\Big(U_n\leftrightarrow\partial B_R(U_n,G_n)\Big)
    	+\bL_n\Big(|B_R(U_n,G_n)|\ge\alpha|V_n|\Big)\,.
	\label{e:chainofineq}
	\end{align}
We will consider separately the two terms on the right-hand side of \eqref{e:chainofineq}. For the first term, Corollary~\ref{c:localevent} directly implies
	\[\lim_{n\to\infty}
	\bL_n\otimes\P_{n,p}\Big(U_n\leftrightarrow\partial B_R(U_n,G_n)\Big)
	=\mu\otimes\P_p\Big(\rho\leftrightarrow\partial B_R(\rho,G)\Big)\,.
	\]
For the second term
 on the right-hand side of \eqref{e:chainofineq}, for any finite $K$ we have
	\[
	\limsup_{n\to\infty} \bL_n\Big(|B_R(U_n;G_n)|\ge\alpha|V_n|\Big)
	\le \lim_{n\to\infty} \bL_n\Big(|B_R(U_n;G_n)|\ge K\Big)
	= \mu\Big( |B_R(\rho;G)| \ge K\Big)\,,
	\]
where the last step is by Lemma~\ref{l:localevent}. 
Since $|B_R(\rho;G)|$ is almost surely finite, the right-hand side tends to zero as $K\to\infty$. The left-hand side does not depend on $K$, so we conclude that in fact
	\[\lim_{n\to\infty} \bL_n\Big(|B_R(U_n;G_n)|\ge\alpha|V_n|\Big)=0\,.\]
Substituting these limits back into \eqref{e:chainofineq} gives
	\[
	\limsup_{n\to\infty}
	\alpha\P_{n,p}\Big(|C_n|\ge\alpha|V_n|\Big)
	\le \mu\otimes\P_p\Big(\rho\leftrightarrow\partial B_R(\rho,G)\Big)\,.
	\]
Since $p<p_c(G)$, the right-hand side tends to zero as $R\to\infty$. The left-hand side does not depend on $R$, so we finally conclude that
	\[\lim_{n\to\infty}\P_{n,p}\Big(|C_n|\ge\alpha|V_n|\Big)=0\,,\]
as claimed.
\end{proof}

\subsection{Supercritical percolation regime}
\label{ss:supercrit}

In this subsection we prove Theorem~\ref{t:main} part (iii), which says in short that $p_c(G_n) \le p_c(G)$. We make use of Proposition~\ref{p:prune} to reduce to the case where the $G_n$ have uniformly bounded average degree. The rest of the proof is adapted from that of \cite[Theorem~1.3]{MR2773031}, where the idea is to first percolate with a smaller probability $p-\epsilon>p_c(G)$, then connect the resulting clusters into a linear-sized component with the remaining probability. While \cite{MR2773031} used the bounded degree assumption to obtain a strong concentration bound for the number of percolation clusters, we will use Chebychev's inequality to obtain a weaker (but sufficient) bound for our more general setting. This is based on the following lemma:

\begin{lem}\label{l:distlem}
Under the assumptions of Theorem~\ref{t:main}, we have for any fixed $R$ that
	\[
	\lim_{n\to\infty}
	\frac{1}{|V_n|^2}
	\bigg|\Big\{ (u,v) \in (V_n)^2 : 
	d(u,v) \le R\Big\}
	\bigg|=0\,.
	\]

\begin{proof} We can bound the quantity of interest by
	\begin{align*}
	r(n)&\equiv\frac{1}{|V_n|^2}
	\bigg|\Big\{ (u,v) \in (V_n)^2 : 
	d(u,v) \le R\Big\}
	\bigg| =
	\sum_{v\in V_n}
	\frac{|B_R(v;G_n)|}{|V_n|^2} \\
	&=\frac{ \E_{\bL_n} |B_R(U_n;G_n)|}{|V_n|}
	\le \frac{K}{|V_n|}
	+ \bL_n\Big(|B_R(U_n,G_n)|\ge K\Big)\,,
	\end{align*}
where the last inequality holds for any finite $K$.
It follows using Lemma~\ref{l:localevent} that
	\[
	\limsup_{n\to\infty}
	r(n) \le \mu\Big(|B_R(\rho,G)|\ge K\Big)\,.
	\]
Since $|B_R(\rho,G)|$ is finite almost surely, the right-hand side tends to zero as $K\to\infty$. The left-hand side does not depend on $K$, so we conclude $r(n)\to0$  as claimed.\end{proof}\end{lem}

The rest of the proof closely follows that of 
\cite[Theorem~1.3]{MR2773031}, which uses a sprinkling argument 
(see also
\cite{MR671140,MR2073175}). Given $p>p_c(G)$, let $p_1>p_c(G)$ and $\epsilon>0$ such that
	\[1-p=(1-p_1)(1-\epsilon)\,.\]
Let $\omega_1$ denote an edge percolation configuration with edge probability $p_1$, and let $\omega_\epsilon$ denote an edge percolation configuration with edge probability $\epsilon$. Then the union of $\omega_1$ and $\omega_\epsilon$ is equidistributed as a $p$-percolation. The first step is to show that a positive proportion of the vertices in $\omega_1$ percolate far and hence are contained in sizable components.

\begin{lem}\label{l:estimate.on.perc.one}
If $p>p_c(G)$, then there exists $\delta>0$ such that for any $R>0$,
	\[\lim_{n\to\infty}
	\P_{n,p}\bigg(
	\frac{1}{|V_n|}
	\bigg|\Big\{
	v \in G_n : v\leftrightarrow \partial B_R(v;G_n)
	\Big\}\bigg|
	\le\delta\bigg)=0\,.
	\]

\begin{proof}
The first part of the proof establishes a positive lower bound on the expectation of the quantity of interest and is essentially identical to the argument in the proof of \cite[Corollary~1.2]{MR4275958}. Given $(G,\rho)$, let $f(G,\rho)$ be the probability under a $p$-percolation that the root percolates:
	\[
	f(G,\rho)
	\equiv
	\P_p\bigg(
		\Big|C(\rho ; \omega(G))\Big|=\infty\bigg)\,,
	\]
where $C(\rho ; \omega(G))$ is the connected component of $\omega(G)$ that contains $\rho$. This is a measurable function on $\GG_\star$, since it is the decreasing limit of the local (hence measurable) functions
	\[
	f_R(G,\rho)
	\equiv
	\P_p\bigg(
		\rho \leftrightarrow \partial B_R(\rho;G)
		\bigg)\,. 
	\]
Since $p> p_c(G)$, the quantity $f(G,\rho)$ must be positive, $\mu$-almost surely. Averaging over $\mu$ gives
	\[
	\mu\otimes\P_p\Big(\rho
		\leftrightarrow \partial B(\rho,R)\Big)
	\ge
	\mu\otimes\P_p\bigg(
		\Big|C(\rho ; \omega(G))\Big|=\infty\bigg) \ge \delta>0
	\]
uniformly over all $R$. It follows by Corollary~\ref{c:localevent} that
	\[
	\lim_{n\to\infty}
	\bL_n\otimes\P_{n,p}
	\Big(U_n \leftrightarrow \partial B_R(U_n;G_n)\Big)
	= \mu\otimes\P_p\Big(\rho
		\leftrightarrow \partial B(\rho,R)\Big)
	\ge\delta\,.
	\]
Let $Y_v=Y_v(R)$ be the indicator that $v$ percolates at least distance $R$ under the $p$-percolation, 
	\[Y_v=\mathbf{1}
		\Big\{
		v \leftrightarrow \partial B_R(v;G_n)
		\Big\}\,.\]
Let $X_n=X_n(R)$ be the sum of $Y_v$ over all $v\in V_n$, so $X_n$ is the quantity of interest. If $\E_{n,p}$ denotes expectation with respect to $\P_{n,p}$, then
for all $n$ large enough we have
	\[\frac{\E_{n,p}X_n}{|V_n|}
	=\bL_n\otimes\P_{n,p}
	\Big(U_n \leftrightarrow \partial B_R(U_n;G_n)\Big)
	\ge \frac{\delta}{2}\,.\]
We next apply the second moment method to show that $X_n$ concentrates, as the method used in \cite{MR2773031} no longer applies due to potential unboundedness of the degrees. Note that $\Cov(Y_u,Y_v)\le1$ always (since $Y_u,Y_v$ are indicators).
If $u$ and $v$ lie at distance more than $2R$ in $G_n$, then $Y_u$ and $Y_v$ are clearly independent, so in this case $\Cov(Y_u,Y_v)=0$. Hence, we have that
	\[
	\frac{\Var_{n,p} X_n}{|V_n|^2}
	\le
	\frac{1}{|V_n|^2} \bigg|\Big\{ (u,v) \in (V_n)^2 : 
	d(u,v) \le 2R\Big\}
	\bigg|
	= o_n(1)\,,
	\]
by Lemma~\ref{l:distlem}. It follows by Chebyshev's inequality that
	\[
	\P_{n,p_1}\bigg(X_n\le\frac{\delta|V_n|}4\bigg)
	\to0\,.
	\]
The claim follows by redefining $\delta$.
\end{proof}
\end{lem}

Once we prune the $G_n$, the remainder of the proof is essentially the same as in \cite{MR2773031}, where the uniform degree bound is only used to bound the number of edges in the graph. For the sake of being self-contained, we reproduce the following lemma from the end of the proof of \cite[Theorem~1.3]{MR2773031}:

\begin{lem}\label{l:sprinkling}
Suppose $G_n$ is an expanding sequence of finite graphs, each with average degree at most $d$. Let $\omega_1,\omega_\epsilon$ be two independent edge percolation configurations on $G_n$, with edge probabilities $p_1$ and $\epsilon$ respectively. Let $X_n(R)$ be the number of vertices in $G_n$ that percolate to distance $R$ under $\omega_1$. Let $C_n$ be the largest connected component in $\omega_1\cup\omega$ and $\delta\in(0,1)$. Then there is a large enough constant $R$ such that
	\[\lim_{n\to\infty}
	\P_{n,p}\bigg(|C_n| \ge \frac{\delta|V_n|}3
	\,\bigg|\, X_n(R)>\delta|V_n|\bigg)
	=1\,.\]
The value of $R$ depends on $d$, $\delta$, $\epsilon$, and the expansion rate, but not on $n$.

\begin{proof}
We condition on $\omega_1$ and use the randomness of $\omega_\epsilon$ only. Consider the connected components of $\omega_1$ of size at least $R$, of which there are at most $m=|V_n|/R$. By the bound on $X_n(R)$, these components cover at least a $\delta$ fraction of the vertices in $G_n$. Let $A,B$ be any partition of these components into two sets, such that
	\beq\label{e:AB.size}
	\min\bigg\{|A|,|B|\bigg\} \ge \frac{\delta |V_n|}3\,.
	\eeq
Let $h$ be a positive constant such that \[\frac12\liminf_{n\rightarrow\infty}h_{\delta/3}(G_n)\ge h.\] By the same reasoning as in the proof of Corollary~\ref{c:bdd.deg.constant.pc}, $A$ and $B$ must be joined by at least 
	\[
	\lambda_n = \frac{h\delta|V_n|}3
	\]
edge-disjoint paths for sufficiently large $n$ , at least half of which must be of length at most
	\[
	\frac{|V_n|d/2}{\lambda_n/2}
	= \frac{3 d}{h\delta} \equiv L\,.
	\]
Thus, the probability that $\omega_\epsilon$ fails to have a path from $A$ to $B$ is at most 
	\[
	(1-\epsilon^L)^{\lambda_n/2}
	\le \exp\bigg\{
	-\frac{\lambda_n}{2} \epsilon^L
	\bigg\}
	\]
Let $\bm{E}$ be the event that $\omega_\epsilon$ has a path from $A$ to $B$ for every partition $(A,B)$ satisfying \eqref{e:AB.size}. Then
	\[
	1-\P(\bm{E})
	\le 2^m\exp\bigg\{
	-\frac{\lambda_n}{2} \epsilon^L
	\bigg\}
	\le \exp\bigg\{
	|V_n|\bigg(\frac1R
		- \frac{h\delta}{6} \epsilon^L
		\bigg)
	\bigg\}\,,
	\]
which tends to zero by taking $R$ a large enough constant.
On the event $\bm{E}$, the configuration $\omega=\omega_1\cup\omega_\epsilon$ contains a component that occupies at least a $\delta/3$ fraction of the vertices in $G_n$, as desired.
\end{proof}
\end{lem}

\begin{proof}[Proof of Theorem~\ref{t:main} part (iii)]
By Proposition~\ref{p:prune}, there exists a expanding sequence of subgraphs $\bar G_n$ of $G_n$ such that each $\bar G_n$ has average degree at most $\E(\deg\rho)+1$ and $\bar G_n\lwc(G,\rho)$. Since $\bar G_n$ is a subgraph of $G_n$, the probability that $\omega(G_n)$ contains a component of linear size is at least the probability that $\omega(\bar G_n)$ contains a component of linear size, i.e. $p_c(\bar{G}_n) \ge p_c(G_n)$. Thus, it suffices to prove the statement for $\bar G_n$.

Let $X_n(R)$ be the number of vertices in $\bar G_n$ that percolate to distance $R$ under $\omega_1$ and $C_n$ be the largest connected component in $\omega(\bar G_n)$. We have that \[\P_{n,p}\bigg(|C_n|\ge\frac{\delta|V_n|}3\bigg)\ge\P_{n,p_1}(X_n(R)>\delta|V_n|)\times\P_{n,p}\bigg(|C_n|\ge\frac{\delta|V_n|}3\,\bigg|\, X_n(R)>\delta|V_n|\bigg)\] for any $R,\delta>0$. Taking $\delta$ to be the constant that results from applying Lemma~\ref{l:estimate.on.perc.one} to $\omega_1(\bar G_n)$ and $R$ to be the constant that results from applying Lemma~\ref{l:sprinkling} to $\bar G_n$ and $\delta$, we have that \[\lim_{n\rightarrow\infty}\P_{n,p_1}(X_n(R)>\delta|V_n|)=\lim_{n\rightarrow\infty}\P_{n,p}\bigg(|C_n|\ge\frac{\delta|V_n|}3\,\bigg|\, X_n(R)>\delta|V_n|\bigg)=1.\] It follows that \[\lim_{n\rightarrow\infty}\P_{n,p}\bigg(|C_n|\ge\frac{\delta|V_n|}3\bigg)=1,\] so we may take $\alpha=\delta/3$.
\end{proof}

{\raggedright
\bibliographystyle{alphaabbr}
\bibliography{percrefs}
}

\end{document}